\documentclass[11pt]{article}
\usepackage[utf8]{inputenc}
\usepackage{multirow}
\usepackage{tikz}
\usetikzlibrary{positioning}
\usepackage[a4paper, margin=2.3cm]{geometry}
\usepackage{amsmath,amssymb,amsthm}
\usepackage{xcolor}
\usepackage{thmtools}
\usepackage{thm-restate}

\usepackage{parskip}
\usepackage{float}
\usepackage[hidelinks,hyperfootnotes=false]{hyperref}
\usepackage{tikz-cd}
\usetikzlibrary{matrix,arrows.meta}

\usepackage{setspace}
\usepackage{makecell}

\usepackage{graphicx}
\usepackage{enumitem}
\usepackage{comment}

\newcommand{\F}{\mathbb{F}}

\usepackage{booktabs,tabularx,needspace}

\newtheorem{theorem}{Theorem}[section]

\newtheorem{corollary}[theorem]{Corollary}

\newtheorem{lemma}[theorem]{Lemma}
\newtheorem{question}[theorem]{Question}
\newtheorem{proposition}[theorem]{Proposition}

\newtheorem*{definition*}{Definition}
\newtheorem*{remark*}{Remark}

\makeatletter
\def\theHtheorem{\thesection.\arabic{theorem}}
\@for\ARevEnv:=remark,corollary,problem,lemma,question,proposition,definition,conjecture,example,claim,assumption\do{
  \expandafter\def\csname theH\ARevEnv\endcsname{\theHtheorem}}
\makeatother

\newcommand{\Fp}{\F_p}
\newcommand{\K}{\F_{p^2}}
\newcommand{\cG}{\mathcal G}
\newcommand{\cT}{\mathcal T}
\newcommand{\OO}{\operatorname O}
\newcommand{\fall}[2]{(#1)_{#2}}
\newcommand{\cinc}{c_{\mathrm{inc}}}

\hypersetup{
  pdftitle={Congruence Classes in Fp2: Sharp Results via an Energy Approach},
  pdfauthor={T. Pham, L. Quang-Ham, S. Senger, D. T. Tran, and B. Xue},
  pdfsubject={Finite-field distance geometry and congruence classes},
  pdfkeywords={finite fields, congruence classes, rigid motions, incidence geometry, distance energy}
}

\title{Congruence Classes in $\F_p^2$: Sharp Results via an Energy Approach}
\author{T. Pham \thanks{Institute of Mathematics and Interdisciplinary Sciences, Xidian University. \newline
\hspace*{0.45cm} Email: {\tt thangphammath@xidian.edu.cn}}
\and L. Quang-Ham \thanks{Vietnam Education Strategy and Policy Institute. ~Email: {\tt hamlq@vnies.edu.vn}}
\and S. Senger \thanks{Department of Mathematics, Missouri State University. Email: {\tt stevensenger@missouristate.edu}}
\and D. T. Tran \thanks{VNU University of Science, Hanoi, Vietnam. ~Email: {\tt tranthedung56@gmail.com}}
\and B. Xue \thanks{Institute of Mathematical Sciences, ShanghaiTech University. ~Email: {\tt xuebq@shanghaitech.edu.cn}}}
\date{}

\begin{document}

\maketitle

\begin{abstract}
Let $\cT_k(E)$ denote the set of congruence classes of ordered
$k$-tuples of pairwise distinct points of $E$. Let
$p\equiv3\pmod4$ be prime. For $E\subset\Fp^2$ with
$3\leq|E|\leq p^{3/4}$, we prove that
$|\cT_3(E)|\gg|E|^{11/6}$; for $4\leq|E|\leq p^{3/4}$, we prove that $|\cT_4(E)|\gg|E|^3/\log|E|$. For every fixed $k\geq5$ and
$k\leq|E|\leq p^{3(k-2)/(3k-4)}$, we prove that
$|\cT_k(E)|\gg_k|E|^{k-1}$. The proofs proceed by bounding the moments of the overlap function of rigid motions. The main inputs are an exact identity involving the distance energy and an incidence bound obtained by viewing rigid
motions as lines over $\Fp(i)$. The bound for $k=4$ is sharp up to a logarithmic factor, while the bounds for $k\geq5$ are sharp up to constants.

\end{abstract}



\section{Introduction}
Let $p$ be an odd prime.  For $x=(x_1,x_2)\in\Fp^2$, write
\[
  \|x\|=x_1^2+x_2^2,
  \qquad d(x,y)=\|x-y\|.
\]
We refer to $d(x,y)$ as the distance between $x$ and $y$.
Two ordered $k$-tuples of points are congruent if there exists an element of the full isometry group $\cG=\OO(2,\Fp)\ltimes\Fp^2$ that sends the entries of one to the corresponding entries of the other. For $E\subset\Fp^2$ and $k\geq2$, let $\cT_k(E)$ denote the set of congruence classes of ordered $k$-tuples of pairwise distinct points of $E$; collinear tuples are included. Throughout this paper, we write $X\ll Y$, or equivalently $Y\gg X$, if $X\leq cY$ for an absolute constant $c>0$, and $X\asymp Y$ if $X\ll Y\ll X$. A subscript, as in $\ll_C$ or $\gg_C$, indicates dependence only on the parameter $C$. All logarithms are natural.

For an integer $k\geq2$, a central problem in finite-field distance
geometry is to determine how many congruence classes of ordered
$k$-point configurations must be determined by a finite set
$E\subseteq\Fp^2$. This problem has two complementary forms. In the dense regime, one asks how large $E$ must be in order that $\cT_k(E)$ contain a positive proportion of all $k$-point congruence classes realizable in the plane. In the sparse regime, where such a positive-proportion conclusion may be impossible, one seeks strong lower bounds for $|\cT_k(E)|$ in terms of $|E|$.

The case $k=2$ is the finite-field Erd\H{o}s--Falconer distance
problem. Representative results include the sum--product approach of Bourgain, Katz, and Tao \cite{BKT}, the Fourier-analytic framework of Iosevich and Rudnev \cite{IR07}, pinned-distance estimates \cite{CEHIK,MPPRS}, and recent advances in even dimensions \cite{PhamShenXue}.
The present paper is concerned with the case $k\geq3$.
Here the study of a $k$-point configuration involves the simultaneous behavior of its $\binom{k}{2}$ pairwise distances. We begin with triangles ($k=3$)
and then turn to general $k$.

For triangles in the dense regime, when $p\equiv3\pmod4$, recall that $\Fp^2$ contains
$\asymp p^3$ triangle congruence classes. The problem is therefore to find conditions guaranteeing
$|\cT_3(E)|\gg p^3$. For
$p\equiv 3\pmod 4$, Bennett, Iosevich, and Pakianathan \cite{BIP} established
this result under the assumption $|E|\gg p^{7/4}$ using rigid motions and
incidence geometry. Bennett, Hart, Iosevich, Pakianathan, and Rudnev
\cite{BHIPR} subsequently improved the threshold to $|E|\geq Cp^{8/5}$ as part
of a group-action approach to counting simplices. At this same threshold,
Pham \cite{PhamFixed} showed, again for $p\equiv 3\pmod 4$, that for every $r\in\Fp^\times$, $E$ determines $\gg p^2$ triangle classes having a side of length $r$. In the same paper \cite[Theorem~1.6]{PhamFixed}, he proved more
generally that the condition $|E|\gg p^{2-2/(2k-1)}$ guarantees a positive
proportion of the $\asymp p^{2k-3}$ congruence classes of
$k$-point configurations in the plane for every
$k\geq 3$. This improves an earlier result due to McDonald in \cite{McDonaldLarge}. When $k=3$, this recovers the
exponent $8/5$ of \cite{BHIPR}. 

For small sets, however, a positive proportion is out of reach: $E$ determines at most $|E|^3$ triangles, which may be far smaller than $p^3$. In this regime, the natural goal is a lower bound on $|\cT_3(E)|$ in terms of $|E|$ itself. This is the direction pursued in the present paper.

Our starting point is the following sharp bound in terms of distances. Write
\[
  \Delta(E)=\{d(x,y):x,y\in E\},
  \qquad D=|\Delta(E)\setminus\{0\}|.
\]
For every odd prime $p$ and every $E\subset\Fp^2$ with $|E|\geq3$, one has
\begin{equation}\label{eq:intro-elementary}
  |\cT_3(E)|\gg |E|D.
\end{equation}
Indeed, for each $r\in\Delta(E)\setminus\{0\}$, fix an ordered pair
$(a_r,b_r)\in E^2$ with $d(a_r,b_r)=r$. An isometry fixing both endpoints of a segment of nonzero length is either the identity or the reflection in the line through them. Consequently, as $x$ ranges over $E\setminus\{a_r,b_r\}$,
the triples $(a_r,b_r,x)$ represent at least $(|E|-2)/2$ congruence classes. Since the first side of each such triple has length $r$, distinct values of $r$ yield disjoint collections of classes, and \eqref{eq:intro-elementary} follows.

The estimate \eqref{eq:intro-elementary} is sharp in general. For example, let
\begin{equation}\label{eq:progression}
  E=\{(j,0):0\leq j<n\},\qquad 2n<p.
\end{equation}
Then $D=n-1$. Indeed, the nonzero distances are $m^2$ for
$1\leq m<n$, and they are distinct because $2n<p$. To bound the
number of classes, write an ordered triple as
$((a_1,0),(a_2,0),(a_3,0))$. The translation
$z\mapsto z-(a_1,0)$ sends it to the congruent triple
$((0,0),(h_2,0),(h_3,0))$, where
$h_2=a_2-a_1$ and $h_3=a_3-a_1$ both lie in
$\{-(n-1),\dots,n-1\}$. Thus every class determined by $E$ has a
representative among at most $(2n-1)^2$ such normalized triples, and $|\cT_3(E)|\leq(2n-1)^2\ll |E|D$.

For comparison, we recall the situation in the real plane. For a finite set $E\subset\mathbb{R}^2$, the distinct distances theorem of Guth and Katz \cite{GK}, combined with the argument above, gives $\gg |E|^2/\log|E|$ triangle classes. By applying the Guth--Katz incidence theorem directly to count pairs of congruent triangles, Rudnev \cite{RudnevTriangles} obtained
the sharp bound $\gg |E|^2$ under the assumption that no line contains more than a fixed proportion of the points. The same bound for arbitrary sets, with collinear triples included, follows by combining this result with the elementary estimate applied to a line containing a large proportion of the
points. Mansfield and Passant \cite{MansfieldPassant} subsequently studied the structure of sets determining only $O(|E|^2)$ triangle classes.

Returning to $\Fp^2$, the pinned distances theorem of Murphy, Petridis, Pham, Rudnev, and Stevens \cite[Theorem~3]{MPPRS} implies that $D\gg |E|^{2/3}$ whenever $|E|\leq p^{4/3}$ and $p\equiv 3\pmod 4$. Substituting this into
\eqref{eq:intro-elementary} yields
\begin{equation}\label{eq:intro-distance-consequence}
  |\cT_3(E)|\gg |E|^{5/3}.
\end{equation}
In this paper, we develop a framework that improves this bound and extends to configurations of $k$ points for general $k$.

\subsection{Main results}\label{subsec:main}

All our main results assume $p\equiv 3\pmod 4$, so that $d(x,y)=0$ only when $x=y$. 
This assumption is essential for all the cardinality bounds below;
see the discussion of the split case.

Our main estimate is expressed in terms of the distance energy, which records the multiplicities of distances rather than their number. For $r\in\Fp^\times$, define
\[
  \nu_E(r)=|\{(x,y)\in E^2:x\ne y,\ d(x,y)=r\}|,
  \qquad Q(E)=\sum_{r\in\Fp^\times}\nu_E(r)^2.
\]
Thus $Q(E)$ counts pairs of ordered segments, each with distinct endpoints, that have the same length.

\begin{theorem}\label{thm:adaptive}
Let $C\geq1$ be fixed. Let $p\equiv3\pmod4$ be prime, and let
$E\subset\Fp^2$ with $|E|\geq3$. Suppose that
\begin{equation}\label{eq:energy-hypothesis}
  Q(E)\leq Cp|E|^2.
\end{equation}
Then, with $D=|\Delta(E)\setminus\{0\}|$,
\[
  |\cT_3(E)|
  \gg_{C}
  \max\left\{|E|D,\ \frac{|E|^{7/2}}{Q(E)^{1/2}}\right\}.
\]
\end{theorem}

The first term is \eqref{eq:intro-elementary}; 
the second is our new contribution. Note that
the hypothesis \eqref{eq:energy-hypothesis} concerns the energy rather than
the cardinality of $E$. The isosceles triangle
estimate of Murphy, Petridis, Pham, Rudnev, and Stevens
\cite[Theorem~4]{MPPRS} gives $Q(E)\ll |E|^{10/3}$ for $|E|\leq p^{4/3}$;
see Lemma~\ref{lem:distance-energy-bound}. For $|E|\leq p^{3/4}$, this
verifies \eqref{eq:energy-hypothesis} with an absolute constant $C$, and
yields the following consequence.

\begin{corollary}\label{cor:eleven-sixths}
Let $p\equiv3\pmod4$ be prime, and let $E\subset\Fp^2$ satisfy
$3\leq |E|\leq p^{3/4}$. Then
\[
  |\cT_3(E)|\gg |E|^{11/6},
\]
with an absolute implied constant.
\end{corollary}
Combining this corollary with \eqref{eq:intro-elementary}, we obtain
\begin{equation}\label{eq:intro-combined-cardinality}
  |\cT_3(E)|\gg\max\{|E|D,\ |E|^{11/6}\}.
\end{equation}
The two terms coincide when $D=|E|^{5/6}$, and the second term in
\eqref{eq:intro-combined-cardinality} improves the exponent $5/3$ in \eqref{eq:intro-distance-consequence} to $11/6$.

The progression example \eqref{eq:progression} shows that the energy term in Theorem~\ref{thm:adaptive} is sharp, up to constants, in the regime
$Q(E)\asymp|E|^3$. For each $1\leq m<n$, the distance $m^2$ is
realized by exactly $2(n-m)$ ordered pairs. Hence
\[
  Q(E)=4\sum_{m=1}^{n-1}(n-m)^2
       =\frac23\,n(n-1)(2n-1)\asymp |E|^3,
\]
so the hypothesis \eqref{eq:energy-hypothesis} holds with an absolute constant, and the energy term is
$|E|^{7/2}Q(E)^{-1/2}\asymp|E|^2$. This matches the true value
$|\cT_3(E)|\asymp|E|^2$. (For this set the term $|E|D$ is
also of order $|E|^2$.)

This example does not, however, show that the exponent $11/6$ in
Corollary~\ref{cor:eleven-sixths} is optimal: the progression has
energy $|E|^3$, whereas the exponent $11/6$ reflects a worse case
$Q(E)\asymp|E|^{10/3}$ allowed by Lemma~\ref{lem:distance-energy-bound}.
We do not know a set attaining both $Q(E)\asymp|E|^{10/3}$ and
$|\cT_3(E)|\asymp|E|^{11/6}$. 

The energy hypothesis \eqref{eq:energy-hypothesis} can be removed
altogether, at the cost of a second term.

\begin{theorem}\label{thm:unconditional}
Let $p\equiv3\pmod4$ be prime, and let $E\subset\Fp^2$ with
$|E|\geq3$. Then
\[
  |\cT_3(E)|
  \gg\min\left\{\frac{|E|^{7/2}}{Q(E)^{1/2}},\
  \frac{p|E|^{11/2}}{Q(E)^{3/2}}\right\}.
\]
\end{theorem}

The ratio of the second term to the first is $p|E|^2/Q(E)$, so the first term is the smaller one exactly when $Q(E)\leq p|E|^2$. Consequently, under \eqref{eq:energy-hypothesis} the minimum is at least $C^{-1}$ times the first term, and Theorem~\ref{thm:adaptive} follows from Theorem~\ref{thm:unconditional} together with
\eqref{eq:intro-elementary}.  

Our next result concerns configurations of more than three points.
For the progression example \eqref{eq:progression}, 
write an ordered $k$-tuple as $((a_1,0),\ldots,(a_k,0))$. Translation by $-(a_1,0)$ gives a
congruent representative $((0,0),(h_2,0),\ldots,(h_k,0))$, where each $h_j=a_j-a_1$ has at most $2|E|-1$ possible values. Consequently,
$|\cT_k(E)|\leq(2|E|-1)^{k-1}<(2|E|)^{k-1}$. Hence no lower bound
valid for all sets can exceed the order $|E|^{k-1}$. The method of
Theorem~\ref{thm:adaptive} attains this order as soon as $k\geq5$.

\begin{theorem}\label{thm:k-points}
Let $p\equiv3\pmod4$ be prime, and let $E\subset\Fp^2$.
If $4\leq|E|\leq p^{3/4}$, then
\[
  |\cT_4(E)|\gg\frac{|E|^3}{\log|E|}.
\]
For every fixed $k\geq5$, if
\[
  k\leq|E|\leq p^{3(k-2)/(3k-4)},
\]
then
\[
  |\cT_k(E)|\gg_k|E|^{k-1}.
\]
The latter bound is sharp up to the implied constant.
\end{theorem}

\subsection{Main ideas of the proof}

A motion $g\in\cG$ moves $E$ to the congruent set $g(E)$, and we
measure how much the two sets overlap by
\[
  m_g=|E\cap g^{-1}(E)|=|\{x\in E: g(x)\in E\}| ,
\]
the number of points of $E$ that $g$ sends to points of $E$. We call $m_g$ the overlap of $g$; it is large when $g$ maps many points of $E$ onto $E$. Write $(m)_j=m(m-1)\cdots(m-j+1)$. If two ordered $k$-tuples of distinct points of $E$ are congruent, then some motion $g$ takes the first to the second, and the second is determined by the first and $g$; in particular $g$ has overlap at least $k$. A Cauchy--Schwarz argument (Lemma~\ref{lem:orbit-energy}) therefore
gives, for every $k\geq2$,
\begin{equation}\label{eq:intro-orbit}
  |\cT_k(E)|\geq\frac{(|E|)_k^2}{\sum\limits_{g\in\cG}(m_g)_k}.
\end{equation}
All our results are obtained by bounding the moments
$\sum_g(m_g)_k$ from above, using two pieces of information about
the overlap function $g\mapsto m_g$: an exact formula for its second moment, and an incidence bound for the number of motions whose
overlap is large.

\emph{The second moment.} Since $p\equiv3\pmod4$, the form $x^2+y^2$ is anisotropic, so any two distinct points of $E$ are at nonzero distance. Two ordered segments of the same nonzero length are related by exactly one direct and one reverse motion. Write $\cG^+$ and $\cG^-$ for the direct and reverse motions, respectively. Hence
(Lemma~\ref{lem:second-moment})
\begin{equation}\label{eq:intro-second-moment}
  \sum_{g\in\cG^+}(m_g)_2
  =\sum_{g\in\cG^-}(m_g)_2=Q(E),
  \qquad \sum_{g\in\cG}(m_g)_2=2Q(E).
\end{equation}
This identity controls motions of moderate overlap. Used alone, it gives $\sum_g(m_g)_k\leq2|E|^{k-2}Q(E)$; for triangles, this only recovers the distance-based exponent $5/3$. A separate estimate is
therefore needed for motions with large overlap.

\emph{Extension-field graph lines and Cartesian-product incidences.}
To count the motions with large overlap, we
identify $\Fp^2$ with $\K=\Fp(i)$, where $i^2=-1$, and write
$\bar z=z^p$. By Lemma~\ref{lem:isometries}, every direct motion is of the form $g(z)=uz+v$ with $u\bar u=1$. To such $g$ we associate the line
\[
  \ell_g=\{(z,w)\in\K^2:w=uz+v\}.
\]
A point $(x,y)\in E\times E$ lies on $\ell_g$ exactly when $y=g(x)$,
so
\[
  |\ell_g\cap(E\times E)|=m_g.
\]
Reverse motions $g(z)=u\bar z+v$ are handled in the same way with the point set $\bar E\times E$. This construction is related to, but different from, the motion--incidence approaches of Elekes and Sharir \cite{ElekesSharir}, Guth and Katz
\cite[Proposition~2.7]{GK}, and Bennett, Iosevich, and Pakianathan
\cite[Lemma~2.2]{BIP}. The key point is that, for every
$E$, the incidence points form the Cartesian products $E\times E$
and $\bar E\times E$. We can therefore use the Cartesian-product
incidence theorem of Stevens and de Zeeuw
\cite[Theorem~4]{StevensDeZeeuw}. Applying it to graph lines with
overlap at least $t$, we obtain (Lemma~\ref{lem:rich-motions})
\begin{equation}\label{eq:intro-rich}
  R_t:=|\{g\in\cG:m_g\geq t\}|\ll\frac{|E|^5}{t^4},
  \qquad t\geq t_0,
\end{equation}
provided that the characteristic condition holds separately for direct and reverse motions.

\emph{Cutoffs and overlap moments.} Write $N=|E|$ and split the motions at an overlap threshold $K$. The exact second moment bounds the contribution below $K$ by $O(K^{k-2}Q(E))$. Above $K$, a dyadic sum of the $t^{-4}$ bound in \eqref{eq:intro-rich} contributes $O(N^5/K)$ for $k=3$, $O(N^5\log N)$ for $k=4$, and
$O_k(N^{k+1})$ for $k\geq5$. Suitable choices of $K$ give
\begin{equation}\label{eq:intro-moments}
  \sum_{g\in\cG}(m_g)_3\ll_C N^{5/2}Q(E)^{1/2},
  \qquad
  \sum_{g\in\cG}(m_g)_4\ll N^{5}\log N,
  \qquad
  \sum_{g\in\cG}(m_g)_k\ll_k N^{k+1}\ \ (k\geq5).
\end{equation}
The first estimate holds when $Q(E)\leq CpN^2$, the second when $N\leq p^{3/4}$, and the third, for fixed $k\geq5$, when
$N\leq p^{3(k-2)/(3k-4)}$. The last range follows from a subfamily
argument that makes the incidence estimate available above
$K\asymp\max\{1,N^{3/2}p^{-1/2}\}$. Substitution into
\eqref{eq:intro-orbit} gives the stated bounds.

\emph{The role of $p\equiv3\pmod4$.} The assumption is used twice.
First, it makes $x^2+y^2$ anisotropic, which is what allows the
nonzero distance energy to control the second moment in
\eqref{eq:intro-second-moment}. Second, it makes $\Fp[X]/(X^2+1)$ a field, so that motions can be written as lines over $\K$. Note that the incidence theorem is applied over $\K$, a field with $p^2$ elements but of characteristic $p$; its hypothesis involves the characteristic. Thus the conditions to verify are
$|E|R_t^\pm\leq\cinc p^2$, rather than conditions with $p^4$.

\emph{The split case.}
When $p\equiv1\pmod4$, the plane contains an isotropic line
$L$: every pair of points of $L$ is at distance zero. The restrictions to $L$ of the isometries preserving $L$ form the full group of invertible affine maps. After identifying $L$ with $\Fp$, every ordered $k$-tuple of distinct points of $L$ is therefore congruent to one of the form $(0,1,\zeta_3,\ldots,\zeta_k)$, and for every $E\subset L$ we have
\[
  |\cT_k(E)|\leq p^{k-2}.
\]
Taking $|E|=\lfloor p^{3/4}\rfloor$ along primes
$p\equiv1\pmod4$ gives $|\cT_3(E)|=O(|E|^{4/3})=o(|E|^{11/6})$ and
$|\cT_4(E)|=O(|E|^{8/3})=o(|E|^3/\log|E|)$. Hence the triangle and
four-point conclusions fail in the split case. For each fixed
$k\geq5$, taking $|E|\asymp p^{3(k-2)/(3k-4)}$ gives
$p^{k-2}=o(|E|^{k-1})$, so the higher-point conclusion also fails. There is also a corresponding obstruction to the method. When $p\equiv1\pmod4$, the algebra $\Fp[X]/(X^2+1)$ splits as
$\Fp\times\Fp$ and is not a field, so the graphs of rigid motions are no longer lines over a field and the Stevens--de Zeeuw point--line
incidence theorem used above does not apply.

\section{Preliminaries}\label{sec:preliminaries}

For a real number $a$ and an integer $j\geq0$, write the falling factorial as
\[
  \fall{a}{j}=a(a-1)\cdots(a-j+1),
  \qquad \fall{a}{0}=1.
\]

\subsection{Rigid motions and overlaps}

Assume that $p\equiv3\pmod4$.  Let $\K=\Fp(i)$, where $i^2=-1$,
and identify $(x_1,x_2)\in\Fp^2$ with $z=x_1+ix_2\in\K$.
Conjugation in $\K$ is given by $\bar z=z^p=x_1-ix_2$, and
\[
  z\bar z=x_1^2+x_2^2.
\]
In particular, the quadratic form $x_1^2+x_2^2$ is anisotropic in that it
vanishes only at $(0,0)$.

Every affine isometry has exactly one of the two forms
\begin{equation}\label{eq:motion-forms}
  g^+_{u,v}(z)=uz+v,
  \qquad
  g^-_{u,v}(z)=u\bar z+v,
  \qquad
  u\bar u=1,\qquad v\in\K.
\end{equation}
Since the incidence method rests on this parametrization, we include
the short proof.

\begin{lemma}\label{lem:isometries}
Let $p\equiv3\pmod4$. Every map of one of the forms
\eqref{eq:motion-forms} is an affine isometry of $\Fp^2$, every
affine isometry is of exactly one of these forms, and distinct pairs
$(u,v)$ give distinct maps. In particular $|\cG|=2p^2(p+1)$.
\end{lemma}

\begin{proof}
We first verify that the maps in \eqref{eq:motion-forms} are
isometries. If $u\bar u=1$, then
\[
  \|uz\|=(uz)\overline{uz}
  =u\bar u\,z\bar z=\|z\|,
  \qquad
  \|u\bar z\|=(u\bar z)\overline{u\bar z}
  =u\bar u\,z\bar z=\|z\|.
\]
Thus the linear maps $z\mapsto uz$ and $z\mapsto u\bar z$
preserve $\|\cdot\|$. Adding a translation does not change
differences, so both affine maps in \eqref{eq:motion-forms} preserve
$d$.

Conversely, let $g$ be an affine isometry, put $v=g(0)$, and
define $h(z)=g(z)-v$. Then $h$ is $\Fp$-linear, fixes the origin, and preserves $\|\cdot\|$. Set $u=h(1)$ and $\xi=h(i)$. Since
$\|1\|=\|i\|=1$, we have
\[
  u\bar u=\|u\|=\|h(1)\|=1,
  \qquad
  \xi\bar\xi=\|\xi\|=\|h(i)\|=1.
\]
The polar form associated with $\|\cdot\|$ is
\[
  B(a,b)=a\bar b+\bar a b=\|a+b\|-\|a\|-\|b\|.
\]
Since $h$ preserves $\|\cdot\|$, it also preserves $B$, i.e. $B(h(a), h(b))=B(a, b)$. As $B(1,i)=0$, it follows that
$B(u,\xi)=u\bar\xi+\bar u\xi=0$. Put $\eta=\xi\bar u$. Then
\[
  \eta+\bar\eta=\xi\bar u+\bar\xi u=0,
  \qquad
  \eta\bar\eta=(\xi\bar\xi)(u\bar u)=1.
\]
Every element of $\K$ with trace zero has the form
$\lambda i$ with $\lambda\in\Fp$. Hence $\eta=\lambda i$, and
$\eta\bar\eta=1$ gives $\lambda^2=1$. Therefore
$\xi=\eta u=\pm iu$, and for $x_1,x_2\in\Fp$ we obtain
\[
  h(x_1+ix_2)=x_1u+x_2\xi=u(x_1\pm ix_2).
\]
Thus $h(z)=uz$ or $h(z)=u\bar z$, and adding back $v$ gives
one of the two forms in \eqref{eq:motion-forms}.

It remains to prove uniqueness. Within either form, $v=g(0)$
and $u=g(1)-g(0)$, so the pair $(u,v)$ is determined by $g$. Moreover, if $uz+v=u'\bar z+v'$ for every $z$, then $z=0$ gives $v=v'$ and $z=1$ gives $u=u'$, whereas $z=i$ gives $ui=-ui$, a contradiction.
Thus the direct and reverse forms are disjoint. Finally, the norm-one group is the kernel of the surjective norm map
$\K^\times\to\Fp^\times$, and therefore has
$(p^2-1)/(p-1)=p+1$ elements. Since $v$ has $p^2$ choices, we obtain
$|\cG|=2p^2(p+1)$.
\end{proof}
We write $\cG^+$ and $\cG^-$ for the direct and reverse motions,
respectively. An ordered segment is an ordered pair of distinct points.
Its length is the distance between its endpoints.
If two ordered segments with distinct endpoints have
the same length, then exactly one direct motion and one reverse
motion take the first segment to the second.  Indeed, if their
difference vectors are $a,b\in\K^\times$, the direct multiplier is
$b/a$ and the reverse multiplier is $b/\bar a$; the translation is
then determined by one endpoint.

For $g\in\cG$, define the overlap multiplicity
\[
  m_g=|\{x\in E:g(x)\in E\}|=|E\cap g^{-1}(E)|.
\]

\subsection{Orbit energy}

The following inequality converts moments of the overlap function
into lower bounds for the number of congruence classes.

\begin{lemma}\label{lem:orbit-energy}
Let $k\geq2$ and $|E|\geq k$. Then
\begin{equation}\label{eq:orbit-energy}
  |\cT_k(E)|
  \geq
  \frac{\fall{|E|}{k}^2}{\displaystyle\sum_{g\in\cG}\fall{m_g}{k}}.
\end{equation}
\end{lemma}

\begin{proof}
For a class $\omega\in\cT_k(E)$, let $s_\omega$ be the number of
ordered $k$-tuples of pairwise distinct points of $E$ that belong to $\omega$. Then $\sum_\omega s_\omega=\fall{|E|}{k}$, and
Cauchy--Schwarz gives
\[
  \fall{|E|}{k}^2
  \leq|\cT_k(E)|\sum_{\omega\in\cT_k(E)}s_\omega^2.
\]
The sum on the right counts ordered pairs $(T,T')$ of congruent
$k$-tuples. Each such pair is related by at least one motion $g$ with $g(T)=T'$, and $T'$ is determined by $T$ and $g$. Hence the sum is at most the number of pairs $(T,g)$ such that $T$ and $g(T)$ are $k$-tuples of pairwise distinct points of $E$, which is
$\sum_g\fall{m_g}{k}$.
\end{proof}

\begin{lemma}[Second momment]\label{lem:second-moment}
If $p\equiv3\pmod4$, then
\begin{equation}\label{eq:second-moment}
  \sum_{g\in\cG^+}\fall{m_g}{2}
  =\sum_{g\in\cG^-}\fall{m_g}{2}=Q(E),
  \qquad
  \sum_{g\in\cG}\fall{m_g}{2}=2Q(E).
\end{equation}
\end{lemma}

\begin{proof}
For either sign, the corresponding sum counts a motion of that sign together with an ordered pair of distinct points of $E$ whose images also lie in $E$. Anisotropy ensures that every such segment has nonzero length. By the observation following \eqref{eq:motion-forms}, each pair of ordered segments of the same length contributes one direct motion and one reverse motion. Summing separately over direct and reverse motions gives the first two identities in \eqref{eq:second-moment};
their sum gives the last.
\end{proof}

For $k=2$, Lemma~\ref{lem:orbit-energy} and
Lemma~\ref{lem:second-moment} give $D\geq|E|^2(|E|-1)^2/(2Q(E))$,
which is Cauchy--Schwarz up to the factor $2$; the results of this
paper are the cases $k\geq3$.

\subsection{Distance energy and incidence estimates}

We use the following isosceles triangle estimate of
Murphy, Petridis, Pham, Rudnev, and Stevens
\cite[Theorem~4]{MPPRS}.  If $E\subset\Fp^2$ and
$|E|\leq p^\frac{4}{3}$, then
\begin{equation}\label{eq:hinge-bound}
  T^*(E)\ll |E|^\frac{7}{3},\footnote{Lewko \cite{LK} recently posted a preprint establishing an optimal point--line incidence bound over arbitrary fields. Combining this bound with the methods developed in \cite{Iosevichetal, Io2}, a routine computation shows that the exponent $7/3$ can be improved to $11/5$.}
\end{equation}
where
\[
  T^*(E)=
  \bigl|\{(a,b,c)\in E^3:
  d(a,b)=d(a,c),\ d(b,c)\ne0\}\bigr|.
\]
This isosceles triangle estimate implies directly the following energy bound via a Cauchy-Schwarz argument.
\begin{lemma}\label{lem:distance-energy-bound}
If $p\equiv3\pmod4$ and $|E|\leq p^\frac{4}{3}$, then
\begin{equation}\label{eq:distance-energy-bound}
  Q(E)\ll |E|^\frac{10}{3}.
\end{equation}
\end{lemma}

\begin{proof}
For $x\in E$ and $r\in\Fp^\times$, let
\[
  n_x(r)=|\{y\in E:d(x,y)=r\}|.
\]
Since $\nu_E(r)=\sum_{x\in E}n_x(r)$, Cauchy--Schwarz gives
\begin{equation}\label{eq:Q-hinge}
  Q(E)
  \leq |E|\sum_{x\in E}\sum_{r\in\Fp^\times}n_x(r)^2.
\end{equation}
The double sum on the right counts triples $(x,y,z)\in E^3$ with
$d(x,y)=d(x,z)\ne0$. The terms with $y=z$ contribute
$|E|(|E|-1)$. If $y\ne z$, anisotropy gives $d(y,z)\ne0$, so the
remaining terms are counted by $T^*(E)$.  Hence \eqref{eq:hinge-bound}
and \eqref{eq:Q-hinge} give
\[
  Q(E)\leq |E|\bigl(T^*(E)+|E|(|E|-1)\bigr)
  \ll |E|^\frac{10}{3}.
\] This completes the proof.
\end{proof}

Let $\mathbb F$ be a field. For a finite point set $P\subset \mathbb F^2$ and a finite set $\mathcal L$ of lines in $\mathbb F^2$, let $I(P,\mathcal L)$ denote their number of incidences between $P$ and $\mathcal{L}$.  We also use the following Cartesian product incidence theorem of Stevens and de Zeeuw \cite[Theorem~4]{StevensDeZeeuw}.

\begin{theorem}[Stevens--de Zeeuw]\label{thm:sdz}
Let $A,B$ be finite subsets of a field $\mathbb F$, with $|A|\leq|B|$, and let $\mathcal L$ be a finite set of lines in $\mathbb F^2$. Suppose that
\[
  |A| |B|^2\leq|\mathcal L|^3.
\]
If $\mathbb F$ has positive characteristic $p$, also suppose that
\[
  |A| |\mathcal L|\leq \cinc p^2,
\]
where $\cinc>0$ is a sufficiently small absolute constant.
Then
\begin{equation}\label{eq:sdz}
  I(A\times B,\mathcal L)
  \ll |A|^\frac{3}{4}|B|^\frac{1}{2}|\mathcal L|^\frac{3}{4}
      +|\mathcal L|.
\end{equation}
The implied constant is absolute.
\end{theorem}

The second hypothesis involves the characteristic of
$\mathbb F$, not its cardinality. We will apply the theorem with
$\mathbb F=\K$, which has $p^2$ elements but characteristic $p$, so the condition to be verified is $|A| |\mathcal L|\leq\cinc p^2$ and not $|A| |\mathcal L|\leq\cinc p^4$. This is the form in which the theorem is stated in \cite{StevensDeZeeuw}.

\section{Proof of the main theorems}\label{sec:proof}

For real $t\geq1$, define
\[
  R_t^\pm=|\{g\in\cG^\pm:m_g\geq t\}|,
  \qquad R_t=R_t^++R_t^-.
\]

\subsection{Rich motions}

\begin{lemma}[Rich motions]\label{lem:rich-motions}
Assume that $p\equiv3\pmod4$ and $E\subset\Fp^2$.  There is an
absolute constant $t_0$ such that the following holds.  If
$t_0\leq t\leq |E|$ and
\[
  |E|R_t^\pm\leq \cinc p^2,
\]
then
\begin{equation}\label{eq:rich-motions}
  R_t^\pm\ll \frac{|E|^5}{t^4}.
\end{equation}
\end{lemma}

\begin{proof}
We use the identification of $\Fp^2$ with $\K$.  For the direct
motions, associate $g^+_{u,v}$ with the line
\[
  w=uz+v
\]
in $\K^2$.  This line contains exactly $m_g$ points of $E\times E$. For the reverse motions, use the same line model with the point set $\bar E\times E$, where $\bar E=\{\bar z:z\in E\}$.
For each fixed sign, distinct motions give distinct lines.

Fix either sign and put $R=R_t^\pm$.  If $R<|E|$, the desired
estimate follows, since $t\leq|E|$ implies
\[
|E|\leq \frac{|E|^5}{t^4}.
\]
Suppose that $R\geq|E|$. The size
condition in Theorem~\ref{thm:sdz} holds because $|E|^3\leq R^3$,
and its characteristic condition is the assumed inequality
$|E|R\leq \cinc p^2$. Here the relevant parameter
is the characteristic $p$ of $\K$, not its cardinality $p^2$.
Applying that theorem over $\K$ gives
\[
 t R
  \leq I(E\times E,\mathcal L)
  \ll |E|^\frac{5}{4}R^\frac{3}{4}+R
\]
for direct motions, and the same estimate with $\bar E\times E$ for reverse motions. For $t\geq t_0$, the last term can be absorbed into the left side. Rearranging gives
\[
R\ll \frac{|E|^5}{t^4}.
\]
\end{proof}

\begin{lemma}\label{lem:subfamily-cutoff}
There are absolute constants $c_0,C_0>0$ with the following
property. Let $p\equiv3\pmod4$ and let $E\subset\Fp^2$. If
$1\leq |E|\leq c_0p$ and
\[
  K=C_0\max\{t_0,2,|E|^{3/2}p^{-1/2}\},
\]
then
\begin{equation}\label{eq:subfamily-cutoff}
  |E|R_t^\pm\leq\cinc p^2
  \qquad\text{for every }K\leq t\leq |E|.
\end{equation}
\end{lemma}

\begin{proof}
Fix a sign and suppose, to the contrary, that
$|E|R_t^\pm>\cinc p^2$ for some $K\leq t\leq |E|$. Put
\[
  M=\left\lfloor\frac{\cinc p^2}{2|E|}\right\rfloor.
\]
The contrary assumption gives
$R_t^\pm>\cinc p^2/|E|\geq2M$, so a subfamily of $M$ such motions is available.
By choosing $c_0$ sufficiently small, we may ensure that
$M\geq |E|$ and $M\geq \cinc p^2/(4|E|)$. Select $M$ motions counted by $R_t^\pm$ and take their graph lines. For direct motions use the point set $E\times E$, and for reverse motions use $\bar E\times E$.
The two hypotheses of Theorem~\ref{thm:sdz} hold because
\[
  |E|^3\leq M^3,
  \qquad |E|M\leq\cinc p^2.
\]
Every selected line contains at least $t$ points of the relevant
Cartesian product, so Theorem~\ref{thm:sdz} gives
\[
  tM\ll |E|^{5/4}M^{3/4}+M.
\]
Choosing $C_0$ sufficiently large allows the last term to be
absorbed. It follows that
\[
  t\ll |E|^{5/4}M^{-1/4}
  \ll |E|^{3/2}p^{-1/2},
\]
contrary to $t\geq K$.
\end{proof}

\subsection{Moments of the overlap function}

Each moment $\sum_g\fall{m_g}{k}$ is estimated by splitting the
motions at a cutoff $K$: those with $m_g<K$ are controlled by the
second moment, and those with $m_g\geq K$ by
Lemma~\ref{lem:rich-motions}. The next lemma carries out the
splitting for a cutoff at which Lemma~\ref{lem:rich-motions} is
known to apply.

\begin{lemma}[Moment decomposition]\label{lem:decomposition}
Let $p\equiv3\pmod4$, let $E\subset\Fp^2$ with $|E|\geq3$, and let
$k\geq3$ be an integer. Let $K\geq\max\{t_0,2\}$ be a real number
such that
\begin{equation}\label{eq:char-hypothesis}
  |E|R_t^\pm\leq\cinc p^2
  \qquad\text{for all real } t \text{ with } K\leq t\leq|E|.
\end{equation}
Then
\begin{equation}\label{eq:decomposition}
  \sum_{g\in\cG}\fall{m_g}{k}
  \leq 2K^{k-2}Q(E)+C_k\,\Sigma_k(K),
\end{equation}
where $C_k$ depends only on $k$ and
\[
  \Sigma_3(K)=\frac{|E|^5}{K},
  \qquad
  \Sigma_4(K)=|E|^5\log|E|,
  \qquad
  \Sigma_k(K)=|E|^{k+1}\quad(k\geq5).
\]
\end{lemma}

\begin{proof}
We split the sum according to whether $m_g<K$ or $m_g\geq K$.

For the first part, note that $\fall{m}{k}\leq m^{k-2}\fall{m}{2}$
for every integer $m\geq0$. Hence, by Lemma~\ref{lem:second-moment},
\begin{equation}\label{eq:low-part}
  \sum_{m_g<K}\fall{m_g}{k}
  \leq K^{k-2}\sum_{g\in\cG}\fall{m_g}{2}
  =2K^{k-2}Q(E).
\end{equation}

For the second part, if $K>|E|$ there is nothing to prove, since
$m_g\leq|E|$ for every $g$. Otherwise let $J\geq0$ be the largest
integer with $2^JK\leq|E|$, and put $t_j=2^jK$ for $0\leq j\leq J$.
Every motion with $m_g\geq K$ satisfies $t_j\leq m_g<2t_j$ for
exactly one $j\leq J$, because $m_g\leq|E|<2^{J+1}K$. For such a
motion $\fall{m_g}{k}\leq m_g^k<2^kt_j^k$, and the number of such
motions is at most $R^+_{t_j}+R^-_{t_j}$. Since
$t_0\leq t_j\leq|E|$
and \eqref{eq:char-hypothesis} holds at $t=t_j$,
Lemma~\ref{lem:rich-motions} gives
$R^\pm_{t_j}\leq C'|E|^5/t_j^4$
with an absolute constant $C'$. Therefore
\begin{equation}\label{eq:high-part}
  \sum_{m_g\geq K}\fall{m_g}{k}
  \leq 2^{k+1}C'\,|E|^5\sum_{j=0}^{J}t_j^{k-4}.
\end{equation}
It remains to evaluate the last sum. For $k=3$ it equals
$K^{-1}\sum_{j\leq J}2^{-j}\leq2/K$. For $k=4$ it equals $J+1$, and $J+1\leq\log_2|E|\leq2\log|E|$ because, using $K\geq2$,
$2^{J+1}\leq2^JK\leq|E|$. For $k\geq5$ it is a geometric series with ratio $2^{k-4}\geq2$, so it is at most
$2t_J^{k-4}\leq2|E|^{k-4}$. In each case \eqref{eq:low-part} and
\eqref{eq:high-part} give \eqref{eq:decomposition}.
\end{proof}

The hypothesis \eqref{eq:char-hypothesis} is verified with the
second moment.

\begin{lemma}[Second-moment control of rich motions]
\label{lem:char-condition}
Let $p\equiv3\pmod4$ and $E\subset\Fp^2$. For every real $t\geq2$,
\begin{equation}\label{eq:R-second-moment}
  R_t^\pm\leq\frac{2Q(E)}{t^2}.
\end{equation}
Consequently, \eqref{eq:char-hypothesis} holds for every real
$K\geq2$ with
\begin{equation}\label{eq:K-condition}
  K\geq\left(\frac{2}{\cinc}\right)^{1/2}
  \frac{(|E|Q(E))^{1/2}}{p}.
\end{equation}
\end{lemma}

\begin{proof}
Every motion counted by $R_t^\pm$ contributes at least
$t(t-1)\geq t^2/2$ to $\sum_{g\in\cG^\pm}\fall{m_g}{2}$, which
equals $Q(E)$ by Lemma~\ref{lem:second-moment}. This gives
\eqref{eq:R-second-moment}. If $K\geq2$ satisfies \eqref{eq:K-condition} and $t\geq K$, then
$|E|R_t^\pm\leq2|E|Q(E)/K^2\leq\cinc p^2$.
\end{proof}

\subsection{The third moment}

Throughout the rest of this section, put
\begin{equation}\label{eq:L0}
  L_0=\max\{t_0,\,2,\,(2/\cinc)^{1/2}\},
\end{equation}
an absolute constant, and write $Q=Q(E)$. Anisotropy gives
$\sum_{r\neq0}\nu_E(r)=|E|(|E|-1)$, so that
\begin{equation}\label{eq:Q-trivial}
  0<Q\leq|E|^2(|E|-1)^2\leq|E|^4.
\end{equation}

\begin{proposition}[Third moment]\label{prop:third-moment}
Let $p\equiv3\pmod4$ be prime, and let $E\subset\Fp^2$ with
$|E|\geq3$. Then
\begin{equation}\label{eq:third-moment-unconditional}
  \sum_{g\in\cG}\fall{m_g}{3}
  \ll|E|^{5/2}Q^{1/2}+\frac{|E|^{1/2}Q^{3/2}}{p}.
\end{equation}
In particular, if $Q\leq Cp|E|^2$ for some $C\geq1$, then
\begin{equation}\label{eq:third-moment-final}
  \sum_{g\in\cG}\fall{m_g}{3}
  \ll_{C}|E|^{5/2}Q^{1/2}.
\end{equation}
\end{proposition}

\begin{proof}
Choose the cutoff
\begin{equation}\label{eq:cutoff}
  K=L_0\max\left\{\frac{|E|^{5/2}}{Q^{1/2}},\
  \frac{(|E|Q)^{1/2}}{p}\right\}.
\end{equation}
By \eqref{eq:Q-trivial}, $|E|^{5/2}Q^{-1/2}\geq|E|^{1/2}\geq1$, so
$K\geq L_0\geq\max\{t_0,2\}$; and $K$ satisfies
\eqref{eq:K-condition} by the choice of $L_0$. Hence
Lemma~\ref{lem:char-condition} shows that
\eqref{eq:char-hypothesis} holds, and Lemma~\ref{lem:decomposition}
with $k=3$ gives
\[
  \sum_{g\in\cG}\fall{m_g}{3}
  \leq2KQ+C_3\frac{|E|^5}{K}.
\]
Using $\max\{a,b\}\leq a+b$ in \eqref{eq:cutoff},
\[
  KQ\leq L_0\left(|E|^{5/2}Q^{1/2}+\frac{|E|^{1/2}Q^{3/2}}{p}\right),
  \qquad
  \frac{|E|^5}{K}\leq\frac{|E|^5Q^{1/2}}{L_0}|E|^{5/2}
  \leq|E|^{5/2}Q^{1/2},
\]
which proves \eqref{eq:third-moment-unconditional}. If
$Q\leq Cp|E|^2$, then
\[
  \frac{|E|^{1/2}Q^{3/2}}{p}
  =|E|^{5/2}Q^{1/2}\cdot\frac{Q}{p|E|^2}
  \leq C|E|^{5/2}Q^{1/2},
\]
and \eqref{eq:third-moment-final} follows.
\end{proof}

\begin{proof}[Proof of Theorems~\ref{thm:adaptive}
and~\ref{thm:unconditional}]
For $|E|\geq3$ we have $\fall{|E|}{3}\geq(2/9)|E|^3$, so
Lemma~\ref{lem:orbit-energy} with $k=3$ gives
$|\cT_3(E)|\gg|E|^6/\sum_g\fall{m_g}{3}$. Inserting
\eqref{eq:third-moment-unconditional} and using
$(a+b)^{-1}\geq\frac12\min\{a^{-1},b^{-1}\}$ yields
Theorem~\ref{thm:unconditional}. Inserting
\eqref{eq:third-moment-final} yields
$|\cT_3(E)|\gg_C|E|^{7/2}Q^{-1/2}$ under
\eqref{eq:energy-hypothesis}, which together with
\eqref{eq:intro-elementary} is Theorem~\ref{thm:adaptive}.
\end{proof}

\begin{proof}[Proof of Corollary~\ref{cor:eleven-sixths}]
Let $C_Q$ be an absolute constant for which
Lemma~\ref{lem:distance-energy-bound} gives $Q\leq C_Q|E|^{10/3}$.
Since $|E|\leq p^{3/4}$, we have $|E|^{4/3}\leq p$, and hence
\[
  Q\leq C_Q|E|^{10/3}\leq C_Qp|E|^2.
\]
Theorem~\ref{thm:adaptive} with the absolute parameter
$C=\max\{1,C_Q\}$ gives
\[
  |\cT_3(E)|
  \gg\frac{|E|^{7/2}}{Q^{1/2}}
  \gg |E|^{7/2-5/3}=|E|^{11/6}.
\]
\end{proof}

\subsection{Higher moments}

\begin{proposition}[Higher moments]\label{prop:higher-moments}
Let $p\equiv3\pmod4$ be prime, and let $E\subset\Fp^2$ with
$|E|\geq3$. Then
\begin{align}
  \sum_{g\in\cG}\fall{m_g}{4}
  &\ll |E|^{5}\log|E|+\frac{|E|Q^{2}}{p^{2}},
  \label{eq:hm4}\\
  \sum_{g\in\cG}\fall{m_g}{k}
  &\ll_k |E|^{k+1}+\frac{|E|^{(k-2)/2}Q^{k/2}}{p^{k-2}}
  \qquad(k\geq5).
  \label{eq:hmk}
\end{align}
\end{proposition}

\begin{proof}
For $k=4$ we use the cutoff $K$ of \eqref{eq:cutoff}, and for
$k\geq5$ the cutoff
\begin{equation}\label{eq:cutoff-k}
  K=L_0\max\left\{|E|^{\frac{k+1}{k-2}}Q^{-\frac{1}{k-2}},\
  \frac{(|E|Q)^{1/2}}{p}\right\}.
\end{equation}
In both cases $K\geq L_0$; for \eqref{eq:cutoff-k} this follows from
\eqref{eq:Q-trivial}, which gives
\[
  |E|^{\frac{k+1}{k-2}}Q^{-\frac1{k-2}}\geq|E|^{\frac{k-3}{k-2}}\geq1.
\]
Both cutoffs satisfy \eqref{eq:K-condition}, so
Lemmas~\ref{lem:char-condition} and~\ref{lem:decomposition} apply.

For $k=4$, Lemma~\ref{lem:decomposition} gives
$\sum_g\fall{m_g}{4}\leq2K^2Q+C_4|E|^5\log|E|$, and
$\max\{a,b\}^2\leq a^2+b^2$ in \eqref{eq:cutoff} gives
\[
  K^2Q\leq L_0^2\left(|E|^5+\frac{|E|Q^2}{p^2}\right).
\]
This proves \eqref{eq:hm4}. For $k\geq5$,
Lemma~\ref{lem:decomposition} gives
$\sum_g\fall{m_g}{k}\leq2K^{k-2}Q+C_k|E|^{k+1}$, and
$\max\{a,b\}^{k-2}\leq a^{k-2}+b^{k-2}$ in \eqref{eq:cutoff-k} gives
\begin{align*}
  K^{k-2}Q
  &\leq L_0^{k-2}\left(|E|^{k+1}Q^{-1}
  +\frac{(|E|Q)^{(k-2)/2}}{p^{k-2}}\right)Q\\
  &=L_0^{k-2}\left(|E|^{k+1}
  +\frac{|E|^{(k-2)/2}Q^{k/2}}{p^{k-2}}\right).
\end{align*}
This proves \eqref{eq:hmk}.
\end{proof}

\begin{proof}[Proof of Theorem~\ref{thm:k-points}]
For the four-point assertion, $|E|\leq p^{3/4}$
gives $p\geq |E|^{4/3}$ and $Q(E)\ll |E|^{10/3}$. Substitution into \eqref{eq:hm4} gives
\[
  \frac{|E|Q(E)^2}{p^2}
  \ll |E|^{1+20/3-8/3}=|E|^5,
  \qquad
  \sum_{g\in\cG}\fall{m_g}{4}\ll |E|^5\log |E|.
\]
If $|E|\geq8$, Lemma~\ref{lem:orbit-energy} yields
\[
  |\cT_4(E)|\gg
  \frac{|E|^8}{|E|^5\log |E|}
  =\frac{|E|^3}{\log |E|}.
\]
For $4\leq |E|<8$, the same bound holds after adjusting the
implied constant.

For the second assertion, fix $k\geq5$ and set
\[
  \alpha_k=\frac{3(k-2)}{3k-4}.
\]
Let $c_0$ be the constant in Lemma~\ref{lem:subfamily-cutoff}.
Since $\alpha_k<1$, there is a constant $p_0(k)$ such that
$p^{\alpha_k}\leq c_0p$ whenever $p\geq p_0(k)$. If $p<p_0(k)$,
then $|E|$ is bounded in terms of $k$, and the result follows from
$|\cT_k(E)|\geq1$ after adjusting the implied constant. We may
therefore assume that $p\geq p_0(k)$, so that $|E|\leq c_0p$.
\[
  K=C_0\max\{t_0,2,|E|^{3/2}p^{-1/2}\},
\]
where $C_0$ is the constant in
Lemma~\ref{lem:subfamily-cutoff}. Since $|E|\leq p^{\alpha_k}<p$,
Lemma~\ref{lem:distance-energy-bound} gives $Q(E)\ll |E|^{10/3}$.
Lemmas~\ref{lem:subfamily-cutoff} and~\ref{lem:decomposition} give
\[
  \sum_{g\in\cG}\fall{m_g}{k}
  \ll_k K^{k-2}Q(E)+|E|^{k+1}.
\]
Since
\[
  K^{k-2}\ll_k
  1+\frac{|E|^{\frac32(k-2)}}{p^{(k-2)/2}},
\]
we have
\begin{align*}
  K^{k-2}Q(E)
  &\ll_k |E|^{10/3}
  +\frac{|E|^{\frac32(k-2)+\frac{10}{3}}}{p^{(k-2)/2}}\\
  &=|E|^{10/3}
  +\frac{|E|^{(9k+2)/6}}{p^{(k-2)/2}}.
\end{align*}
The assumed range $|E|\leq p^{\alpha_k}$ is equivalent to
\[
  |E|^{(3k-4)/6}\leq p^{(k-2)/2}.
\]
Also, $|E|^{10/3}\leq |E|^{k+1}$ for $k\geq5$, so both terms
in the last display are $O_k(|E|^{k+1})$. Hence
\[
  \sum_{g\in\cG}\fall{m_g}{k}\ll_k |E|^{k+1}.
\]
If $|E|\geq2k$, Lemma~\ref{lem:orbit-energy} now gives
\[
  |\cT_k(E)|\gg_k
  \frac{|E|^{2k}}{|E|^{k+1}}=|E|^{k-1}.
\]
If $k\leq |E|<2k$, the same bound holds after adjusting the
implied constant. Sharpness follows from the progression
\eqref{eq:progression}, for which
$|\cT_k(E)|\leq(2|E|)^{k-1}$.
\end{proof}

\subsection{Sharpness of the third-moment estimate}
\label{subsec:sharpness}

The purpose of this subsection is twofold. We first show that the conditional third-moment estimate \eqref{eq:third-moment-final} is sharp, up to constants, in the energy regime
$Q(E)\asymp |E|^3$. We then compare the rich-motion estimate with the second-moment bound and identify the range in which the incidence input gives an improvement.

Put $N=|E|$ and let $E$ be the progression
\eqref{eq:progression}. Then $Q(E)\asymp N^3$. Every motion with
$m_g\geq2$ maps two distinct points of the horizontal axis to two
distinct points of that axis, and hence maps the axis onto itself. By Lemma~\ref{lem:isometries}, it belongs to one of the four families $z\mapsto\pm z+\tau$ and $z\mapsto\pm\bar z+\tau$,
with $\tau\in\Fp$. Represent $\tau$ by the integer in $(-p/2,p/2)$. For the translations $z\mapsto z+\tau$, the overlap is $N-|\tau|$ when $|\tau|<N$ and zero
otherwise; the condition $2N<p$ rules out wrap-around. In each of the four families, the overlap is nonzero for $O(N)$ values of $\tau$ and is at most $N$. Consequently,
\[
  N^4\ll\sum_{|\tau|\leq N/2}\fall{N-|\tau|}{3}
  \leq\sum_{g\in\cG}\fall{m_g}{3}\ll N^4.
\]
Since $N^{5/2}Q(E)^{1/2}\asymp N^4$ and
$Q(E)\ll N^3\leq pN^2$, the conditional third-moment estimate
\eqref{eq:third-moment-final} is sharp, up to constants, in the
energy regime $Q(E)\asymp N^3$. This example does not determine the optimal bound when $Q(E)$ is closer to $N^{10/3}$, and therefore does not settle the exponent $11/6$ in Corollary~\ref{cor:eleven-sixths}.

The role of the incidence theorem can also be isolated. Without
Lemma~\ref{lem:rich-motions}, the second-moment identity gives only
\[
  \sum_g\fall{m_g}{3}\leq2NQ(E),
  \qquad
  |\cT_3(E)|\gg\frac{N^5}{Q(E)}.
\]
Together with $Q(E)\ll N^{10/3}$, this recovers
$|\cT_3(E)|\gg N^{5/3}$ when $N\leq p^{4/3}$. On the other hand,
the second moment gives $R_t^\pm\ll Q(E)/t^2$, whereas
Lemma~\ref{lem:rich-motions} gives $R_t^\pm\ll N^5/t^4$, subject to its characteristic condition. Up to constants, the latter is
stronger for $t\gg N^{5/2}Q(E)^{-1/2}$, and summing this upper tail yields the term $N^{5/2}Q(E)^{1/2}$. This improves the elementary bound $NQ(E)$ when $Q(E)\gg N^3$. The incidence theorem does not control motions of small overlap, so it cannot replace the energy input; in particular, motions with $m_g=2$ are invisible to a rich-motion estimate.

\section{Open questions}\label{sec:open-questions}

We conclude with two questions suggested by our results.
\begin{question}
How far can the ranges in Corollary~\ref{cor:eleven-sixths}
and Theorem~\ref{thm:k-points} be extended? In particular, what lower bounds hold when
$$
  |E|>p^{3/4}\quad\text{for }k=3,4,
$$
or when
$$
  |E|>p^{3(k-2)/(3k-4)}\quad\text{for }k\geq5?
$$
\end{question}

\begin{question}
Suppose that $p\equiv1\pmod4$. Under what additional geometric
conditions on $E$, such as suitable non-concentration on isotropic
lines, do analogues of Corollary~\ref{cor:eleven-sixths} and
Theorem~\ref{thm:k-points} hold?
\end{question}

\end{document}